\documentclass[12pt]{article}
\usepackage{amssymb,amsmath,amsthm,amscd,latexsym}
\usepackage{mathrsfs}
\usepackage{indentfirst}
\usepackage{enumitem}
\usepackage{anysize}
\usepackage{hyperref}
\usepackage{color}

\renewcommand{\paragraph}{\roman{paragraph}}
\input cyracc.def

\newcommand{\R}{\mathbb{R}}
\newcommand{\Z}{\mathbb{Z}}

\newcommand{\abs}[1]{\left|#1\right|}
\newcommand{\F}{\mathbb{F}}

\newcommand{\whom}{\mathrm{w}_{\mathrm{hom}}}
\newcommand{\dhom}{\mathrm{d}_{\mathrm{hom}}}

\newtheorem{theorem}{Theorem}
\newtheorem{corollary}{Corollary}

\newtheorem{lemma}{Lemma}
\newtheorem{proposition}{Proposition}
\newtheorem{conjecture}{Conjecture}

\theoremstyle{definition}
\newtheorem{definition}{Definition}

\newcommand{\cayley}{\mathrm{C}}

\newcommand{\SWRGparam}{s}

\begin{document}
\title{\bf Three-weight codes over rings  \\and\\ strongly walk regular graphs\thanks{This research is supported by National Natural Science Foundation of China (61672036), Excellent Youth Foundation of Natural Science Foundation of Anhui Province (1808085J20),
Technology Foundation for Selected Overseas Chinese Scholar, Ministry of Personnel of China (05015133) and
Key projects of support program for outstanding young talents in Colleges and Universities (gxyqZD2016008).}}

\author{
Michael Kiermaier\thanks{Department of Mathematics, University of Bayreuth, Bayreuth, Germany},
Sascha Kurz$^\ddagger$,%%\thanks{Department of Mathematics, University of Bayreuth, Bayreuth, Germany},%%\newline
Minjia Shi\thanks{School of Mathematical Sciences of Anhui University, Anhui, 230601, P. R. China}, and
Patrick Sol\'e\thanks{I2M, CNRS, Aix-Marseille Univ, Centrale Marseille, Marseille, France}
}

\date{}

\maketitle
\noindent
{\bf Abstract:} { We construct strongly walk-regular graphs as coset graphs of the duals of  codes with three non-zero homogeneous weights over $\Z_{p^m},$ for $p$ a prime,
and more generally over chain rings of depth $m$, and with a residue field of size $q$, a prime power.
Infinite families of examples are built from Kerdock and generalized Teichm\"uller codes. As a byproduct, we give an alternative proof that the Kerdock code is nonlinear.}


\noindent
{\bf MSC (2010):} Primary 05 E 30,  Secondary 94 B 05
\section{Introduction}\label{s:intro} Since the seminal article of Delsarte \cite{De},
there is a well-known interplay between two-weight codes and strongly regular graphs (SRG) via the coset graph of the dual code \cite{BCN,BH}.
Strongly walk-regular graphs (SWRG) were introduced in \cite{DO} as a generalization of strongly regular graphs. Instead of a regularity condition bearing on paths
of length $2$, the notion of SWRG demands a regularity on paths of length $\SWRGparam>1$. Specifically, a graph is $\SWRGparam$-SWRG if there
are three integers $(\lambda,\mu,\nu)$ such that the number of paths of length $\SWRGparam$ between any two vertices $x$ and $y$ is
\begin{itemize}
\item $\lambda$ if $x$ and $y$ are connected;
\item $\mu$ if $x$ and $y$ are disconnected;
\item $\nu$ if $x=y.$
\end{itemize}
This is reminiscent of strongly regular graph (SRG) a regular graph on $\nu$ vertices such that the number of paths of length $2$ between any two vertices $x$ and $y$ is
\begin{itemize}
\item $\lambda$ if $x$ and $y$ are connected;
\item $\mu$ if $x$ and $y$ are disconnected;
\item $\kappa$ if $x=y$.
\end{itemize}

Note that, by definition, strongly regular graph are $s$-SWRG for all $s>1$.
In a previous paper \cite{SS}, the authors connected together the two notions of triple sum sets (TSS) and SWRG over finite fields. In the present
paper we extend this correspondence to a correspondence between TSS's over certain finite rings and SWRG's via the notions of syndrome graphs and homogeneous
weights. This leads us to study the sum of the weights of three-weight codes. The first class of rings we consider is that of $\Z_{p^m},$ for $p$ a prime, when
the homogeneous weight can described explicitly. Building on the study in \cite{S+}, we give a sufficient condition to construct SWRGs from the dual code of a
three-weight code. This condition bears on the sum of the three weights in question. In concrete examples, we focus on $\Z_4$-codes which have been studied
extensively since \cite{HKCSS}. The Gray map, which preserves MacWilliams duality, allows us to write the four Pless power moments for the binary image code,
and to classify the relevant three-weight codes in short length. Further, a class of irreducible cyclic codes \cite{S} related to the Kerdock code \cite{HKCSS}
gives an alternative proof of the nonlinearity of the binary Kerdock code. The second class of rings considered is that of chain rings of given depth and residue
field. The $p$-ary three-weight codes introduced in \cite{Sh3}, and the generalized Teichm\"uller codes of \cite{K-Diss} provide infinite families of three-weight codes leading to SWRG's.

The material is organized as follows. The next section collects some prerequisite on, successively, homogeneous weights, triple sum sets, and coset graphs of codes over rings.
Section~\ref{sec_main} derives the main results. Section~\ref{sec_z4_codes} focuses on $\Z_4$ codes and gives examples in short lengths as well as an infinite family
related to the Kerdock codes. Section~\ref{sec_cr} describes how the main results extend when the alphabet is a finite chain ring, and gives an infinite family related to generalized Teichm\"uller codes.
\section{Background}\label{sec:back}
\subsection{Homogeneous weight}
\label{subsec_homogeneous weight}
For simplicity, let $R$ be a finite commutative ring.
A \emph{linear code} $C$ over $R$ of length $n$ is an
$R$-submodule of the standard $R$-module $R^n$. The
\emph{dual code} of $C$ is defined as $C^\bot=\{x\in R^n\;x\cdot C=0\}$,
where $x\cdot y=x_1y_1+\dots+x_ny_n$ denotes the standard inner
product on $R^n$.
\begin{definition}\label{dfn:whom}
Let $R$ be a finite ring. A function $w\colon R\to\R$
is called a \emph{homogeneous weight}, if $w(0)=0$ and
\begin{enumerate}
\item if $Rx=Ry$ then $w(x)=w(y)$ for all $x,y \in R$;
\item there exists a real number $\gamma\neq 0$ such that
$\sum_{y\in Rx} w(y)=\gamma \vert Rx \vert$ for all $x \in R\setminus\{0\}$.\\
\end{enumerate}
\end{definition}
It is known from the work in
\cite{ioana-werner97,greferath-schmidt99,st:homippi} that a
homogeneous weight exists for all finite commutative rings $R,$ and is uniquely determined up to the
normalization factor $\gamma$. In the sequel we will denote the
homogeneous weight by $\whom$, assuming that $\gamma$ has been fixed.
As usual we extend $\whom$ to a weight function on $R^n$ by
$$\whom(x_1,\dots,x_n)=\sum_{i=1}^{n}\whom(x_i).$$
The set $R^n$ then forms a metric space under the \emph{homogeneous
  distance} $\dhom$ defined by $\dhom(x,y)=\whom(x-y)$.  A code
$C\subseteq\R^n$ is called a \emph{(homogeneous) three-weight code} with
weights $w_1<w_2<w_3$ if every non-zero codeword has homogeneous weight
$w_1$, $w_2$, or $w_3$, and the three weights actually occur. If the code
$C$ is linear then on $C\times C$ the homogeneous distance takes only these
three values and zero.

A matrix $G\in R^{\ell\times n}$ is called a \emph{generator matrix}
of the linear code $C$ if the rows of $G$ generate $C$ as an
$R$-module.

\begin{definition}
Let $C$ have $\ell\times n$ generator matrix $G=[g_1\vert \dots
\vert g_n]$. The code $C$ is called:
\begin{enumerate}
\item \emph{proper} if $\whom(c)=0$ implies
  $c=0$ for all $c \in C$;
\item \emph{regular} if $\{x\cdot g_i;x\in R^{\ell}\}=R$ for $i=1,\dots,n$;
\item \emph{projective} if $Rg_i\neq Rg_j$ for any pair of distinct
  coordinates $i,j \in \{1,\dots,n\}$.
\end{enumerate}
\end{definition}

In other words, a code is regular if every column contains at least one unit element. We
extend this notion also to a single column vector.

In the sequel we shall be interested mostly in the case
$R=\Z_{p^m}$, the ring of integers modulo a prime power $p^m$.
In this case the homogeneous weight is given by
\begin{equation*}
\whom(x) =
\begin{cases}
    0  &  \text{if $x = 0$},\\
    p^{m-1}   & \text{if $0\neq x\in p^{m-1}\Z_{p^m}$},\\
    (p-1)p^{m-2}  & \text{otherwise}.
\end{cases}
\end{equation*}
The chosen normalization constant is $\gamma=(p-1)p^{m-2}$, which is
the most convenient for our purposes. For $m\ge 2$ the homogeneous weight is
always an integer. Note that every code $C$ over $R=\Z_{p^m}$ is proper, while
there are some non-proper chain rings. Moreover, a code that is regular and projective
iff the minimum distance $d^\bot$ of its dual code is at least $(2p-1)\cdot p^{m-2}$.\footnote{The support
of a dual codeword that is not excluded by the conditions regular and projective has cardinality at least three,
so that the weight is strictly larger than $2(p-1)p^{m-2}$, i.e., at least $(2p-1)p^{m-2}$. For the other direction
note that $(2p-1)p^{m-2}>p^{m-1}$ and $(2p-1)p^{m-2}>2\cdot (p-1)p^{m-2}$.} So, for $p=m=2$ regular and projective
is equivalent to $d^\bot\ge 3$.

\subsection{Triple sum and $\mathbf{s}$-sum sets}
\label{subsec_tss}

We extend the definition of triple sum sets given in \cite{CW} from finite fields to finite commutative rings $R$ with a unit group $R^\times.$
The set $\Omega \subseteq R^k$ is a \emph{triple sum set} (TSS) if
it is stable by scalar multiplication by units $\in R^\times$ and if there are integer constants
$\sigma_0$ and $\sigma_1$ such that a non-zero $h\in R^k$ can be written as
$$h=x+y+z,\,\mbox{with}\, x,y,z \in \Omega,$$
\begin{itemize}
\item $\sigma_0$ times if $h\in \Omega,$
\item $\sigma_1$ times if $h\in R^k \setminus\Omega$.
\end{itemize}
This can generalized immediately to the notion of an \emph{$s$-sum set} by replacing the equation $h=x+y+z$ by $h=\sum\limits_{i=1}^sx_i$ where each $x_i\in \Omega$.
%% The set $\Omega \subseteq R^k$ is an $s$-sum set if
%% it is stable by scalar multiplication by units $\in R^\times$ and if there are integer constants
%% $\sigma_0$ and $\sigma_1$ such that a non-zero $h\in R^k$ can be written as
%% $$h=\sum\limits_{i=1}^sx_i ,\,\mbox{with}\,x_i \in \Omega,$$
%% \begin{itemize}
%% \item $\sigma_0$ times if $h\in \Omega,$
%% \item $\sigma_1$ times if $h\in A^k \setminus\Omega$.
%% \end{itemize}
%%%%%%%%%%%%%%%%%%%%%%%%%%%%%%%
\subsection{Syndrome graph}
\label{subsec_syndrome_graph}

An {\em eigenvalue} of a graph $\Gamma$ (i.e., an eigenvalue of its adjacency
matrix) is called a \emph{restricted eigenvalue} if there is a
corresponding eigenvector which is not a multiple of the all-one vector
\textbf{1}. Note that for an $\eta$-regular connected graph, the
restricted eigenvalues are simply the eigenvalues different from $\eta$.

\begin{definition}\label{2}
  Let $T$ be a finite abelian group and $S\subseteq T$ a subset
  satisfying $S=-S$ and $0_T\notin S$. The
  corresponding {\em Cayley graph} $\cayley(T,S)$ has vertex set equal to
  $T$; two vertices $g, h \in T$ are adjacent in $\cayley(T,S)$ iff $g-h\in S$.
\end{definition}
The conditions imposed on $S$ ensure that $\cayley(T,S)$ is a simple
graph. The graph $\cayley(T,S)$ is regular of degree $\abs{S}$, and it is
connected iff $S$ generates the group $T$.
We use the same idea as in \cite{CK} to associate with any linear
code $C$ over $\Z_{p^m}$ a certain Cayley graph, called the \emph{syndrome
  graph} of $C$ and denoted by $\Gamma(C)$.

Let $H=[h_1\mid \dots\mid h_n]\in(\Z_{p^m})^{\ell\times n}$ be a
parity-check matrix of $C$.
The vertex set of $\Gamma(C)$ is
$V=\bigl\{Hx;x\in(\Z_{p^m})^n\bigr\}$, the column space of $H$, which
is isomorphic to the dual code $C^\bot$.  The elements of $V$ are called
\emph{syndromes} of $C$.  Two syndromes $Hx$, $Hy$
are adjacent in $\Gamma(C)$ if
they differ by a unit-multiple of a column of $H$\,:
\begin{equation}
  \label{eq:Gamma(C)}
  Hx\sim Hy\;:\iff\;H(x-y)=uh_i\quad\text{for some $1\leq i\leq n$ and
    $u\in\Z_{p^m}^\times$}.
  \end{equation}
It is obvious that $\Gamma(C)$ is the Cayley graph of $V$
corresponding to the generating $S=\{uh_i;u\in \Z_{p^m}^\times,1\leq
i\leq n\}$. Since $0\notin S$ and $-S\subseteq S$, $\Gamma(C)$
is simple. Moreover, $\Gamma(C)$ is
regular of degree $\vert S \vert=p^{m-1}(p-1)n$ and has
$\frac{p^{mn}}{\vert C \vert}$ vertices. The latter follows from
$\abs{C}\abs{C^\bot}=\abs{(\Z_{p^m})^n}=p^{mn}$.

As defined, $\Gamma(C)$ depends on the particular choice of $H$. The
terminology ``syndrome graph of $C$'' is justified, however, since
$\Gamma(C)$ has another representation as a ``coset graph'' of $C$,
which shows that its isomorphism type is determined by $C$.
It is well-known that there is a one-to-one correspondence between
syndromes of a code and its cosets \cite{SAN}, given by $x+C\mapsto
Hx$. The condition in \eqref{eq:Gamma(C)} can be restated as
$H(x-y-ue_i)=0$, where $e_i$ denotes the $i$th standard unit vector in
$(\Z_{p^m})^n$. Thus $\Gamma(C)$ is isomorphic to the graph whose
vertices are the cosets of $C$ and in which two vertices $x+C$, $y+C$
are adjacent if their difference is a coset of the form $ue_i+C$.
% thus this graph can be defined on the cosets of $C$. Suppose that
% $xH^T \sim yH^T$ in $\Gamma(C)$. Hence
% $(x-y)H^T=uh_i$ for some $u \in \Z_{p^k}^\times$ and
% $i=1,\dots,n$. Consider $a=(a_1,\dots,a_n)$ such that $a_i=u$ and
% for $j\neq i$, $a_j=0$. So $(x-y)H^T=aH^T$. Therefore two cosets
% $x+C$ and $y+C$ are adjacent if and only if their difference is a
% coset $a+C$,
% where $a=(a_1,\dots,a_n)$ satisfies $a_i=u$ and $a_j=0$ for $j\neq
% i$.
This is the Cayley graph of the quotient group
$(\Z_{p^m})^n/C$ with generating set
$$S^*=\{a+C;\exists i\text{ such that $a_i=u\in\Z_{p^m}^\times$ and
  $a_j=0$ for $j\neq i$}\}.$$

Now, we recall the relation
between the weight distribution of a linear code over $\Z_{p^m}$ and
the eigenvalues of the syndrome graph of its dual code. This extension of  Lemma 3.4 in \cite{CK} was derived in \cite{S+}
Theorem~4.1 and its proof is omitted here.
\begin{theorem}\label{thm_eigenvalues}
  Suppose that $C$ is a regular, projective linear code over $\Z_{p^m}$ with
  homogeneous weights $w_i$ and corresponding weight distribution
  $A_i=\abs{\{x\in C;\whom(x)=w_i\}}$. Then
  the eigenvalues of $\Gamma(C^\bot)$ are $n(p-1)p^{m-1}-pw_i$ with
  multiplicity $A_i$.
\end{theorem}

The {\em coset  graph with $b$ loops} $\Gamma_C^b$ of a code $C$ is obtained from $\Gamma_C$ by adding $b$ loops around every vertex. The following result is immediate from the previous theorem. The proof is omitted.
{\corollary \label{extcor} If $C$ is a $\Z_{p^m}$-code, %%of minimum distance at least three,
with dual weight distribution $[\langle i,A_i\rangle]$, such that $C^\bot$ is regular and projective, then the spectrum of $\Gamma_C^b$ is $\{(b+n(p-1)p^{m-1}-pi)^{A_i}\}$.
Thus $A_i$ is the frequency of the homogeneous weight $i$ in $C^\bot$ and the multiplicity of the eigenvalue $b+n(p-1)p^{m-1}-pi$.}

%%%%%%%%%%%%%%%%%%%%%%%%%%%%%%%%%%%%%%%%%%%%%%%%%%%%%%%%%%%%%%%%%
\section{Main Results}
\label{sec_main}
Let $\overline{R^k}$ denote the set of all regular vectors in $R^k$, i.e., $$\overline{R^k}=\left\{\omega\in R^k\,:\, \{\omega\cdot x\,:\,x\in R^ k\}=R\right\}.$$ If
$\Omega\subseteq \overline{R^k}$, we denote by $C(\Omega)$ the regular and projective code of length $n=\frac{\vert \Omega\vert}{|R|-1}$ obtained as the kernel of the
$k \times n$ matrix $H$ with columns the projectively non-equivalent nonzero elements of $\Omega$. Thus $H$ is the check matrix of $C(\Omega)$.
To simplify notation, let $\Gamma_{C(\Omega)}=\Gamma(\Omega)$ and $\Gamma_{C(\Omega)}^b=\Gamma(\Omega)^b$, where we also write $\Gamma(\Omega)^0$ for $\Gamma(\Omega)$.

The following result is immediate from the definitions in the introduction, and will be used repeatedly and implicitly in the rest of the paper.

{\theorem \label{equiv} The following are equivalent
\begin{itemize}
\item $\Omega$ (resp $\Omega \bigcup 0$) is an $s$-sum set,
\item $\Gamma(\Omega)$ (resp. $\Gamma(\Omega)^1$) is an $s$-SWRG.
\end{itemize} }

\begin{proof} Observe that $\Gamma(\Omega)$ is a Cayley graph on the group $(R^k,+)$ with generators the columns of $H$. Thus the definition of an $s$-sum set can
be regarded as a statement on the number of paths of length $s$ between $0$ and a column $h$ of $H$.
\end{proof}

The following result comes from the well-known connection between SRGs and two-weight codes. It is described for codes over fields in \cite{BCN,BH}, and for codes over rings in \cite{Byrne1}.

{\theorem \label{SRG} If $C(\Omega)^\bot$ is a two-weight code, then both $\Omega \bigcup 0$ and $\Omega$ are $s$-sum sets for all $s>1.$}

\begin{proof} If $C(\Omega)^\bot$ is a two-weight code, it is well-known that $\Gamma(\Omega)$ is an SRG \cite{Byrne1}. The result follows by Theorem \ref{equiv}, and the observation, made in the introduction,  that SRGs are $s$-SWRG for all $s>1.$
\end{proof}

The following result was observed in \cite{HRG} Theorem III.6.

{\theorem \label{HRGnew} If $\Omega$ or $\Omega \bigcup 0$ is an $s$-sum sets, then $C(\Omega)^\bot$ has at most three nonzero weights. Furthermore,  if $s$ is even, then $C(\Omega)^\bot$ is an $s$-weight code.}

We come to the coding-theoretic characterization of $3$-sum sets which is contained for finite fields  in \cite[Th. 2.1]{CW}. The consideration of syndrome graphs with loops might seem artificial but is justified
by the existence of several non-trivial examples, as evidenced by Section~\ref{sec_z4_codes}.

The relation with the $3$-SWRG property and the eigenvalues of a connected regular graph was obtained in \cite{DO} Proposition~4.1:
{\theorem \label{vanDamProp41} Let $\Gamma$ be a connected regular graph with four distinct eigenvalues $k>\theta_1>\theta_2>\theta_3$, then $\Gamma$
is a $3$-SWRG iff $\theta_1+\theta_2+\theta_3=0$.
}

In our setting the formulas for the eigenvalues in Theorem~\ref{thm_eigenvalues} and Corollary~\ref{extcor} give:
{\theorem \label{CWnew} Assume that $C(\Omega)^\bot$ is of length $n$ and has three {nonzero} weights $w_1<w_2<w_3.$ Let $b$ be any non-negative integer.
\begin{itemize}
\item $\Omega$ is a TSS iff $w_1+w_2+w_3=3n(p-1)p^{m-2}$;
\item $\Omega \bigcup 0$ is a TSS iff $w_1+w_2+w_3=3\frac{1+n(p-1)p^{m-1}}{p}$;
\item $\Gamma_{C(\Omega)}^b$ is a $3$-SWRG iff $w_1+w_2+w_3=3\frac{b+n(p-1)p^{m-1}}{p}$;
\end{itemize}
}
\begin{proof} For the first two cases we apply Theorem~\ref{equiv} to conclude the equivalence to $\Gamma_{C(\Omega)}^b$ being a $3$-SWRG for $b=0$
and $b=1$, respectively. Combining Corollary~\ref{extcor} with Theorem~\ref{vanDamProp41} gives the third case, which contains the first two cases for $b=0$
and $b=1$, respectively.
\end{proof}
Note that if $m\ge 2$ and $p\neq 3$, then the second case of Theorem~\ref{CWnew} is impossible. More generally, for the third case
of Theorem~\ref{CWnew}  the assumptions $m\ge 2$ and $p\neq 3$ imply that $b$ has to be divisible by $p$. Then, the condition may be rewritten as
$w_1+w_2+w_3\ge 3n(p-1)p^{m-2}$ and $w_1+w_2+w_3\equiv \pmod 3$.

{\bf Example 1:} If $p=m=2$ and $n=4$, then the possible triples of weights in the first case of Theorem~\ref{CWnew} %%, i.e., $b=0$,
are, with obvious notation, $\{3 4 5,\,
2 4 6,\,
1 5 6,\,
2 3 7,\,
1 4 7,\,
1 3 8\}$.
We will see in Subsection~\ref{subsec_z4_codes_small_length} that none of these possibilities can be attained by a $\Z_4$-code and the smallest possible
length for $p=m=2$ and $b=0$ is indeed $n=6$.

Proposition~4.3 in \cite{DO} gives an explicit diophantine equation bearing on the eigenvalues of $\Gamma$ to be an $s$-SWRG:
{\theorem Let $\Gamma$ be a connected regular graph with four distinct eigenvalues $k>\theta_1>\theta_2>\theta_3$ and $s\ge 3$. Then $\Gamma$ is an $s$-SWRG
iff
\begin{equation}
  (\theta_2-\theta_3)\theta_1^s+(\theta_3-\theta_1)\theta_2^s+(\theta_1-\theta_2)\theta_3^s=0.\label{eq_diophantine_vandam}
\end{equation}
}

%% The following result gives an explicit diophantine equation bearing on the weights of $C(\Omega)^\bot$ for $\Omega$ to be an $s$-sum set.

We remark that for $s=3$ Equation~(\ref{eq_diophantine_vandam}) is equivalent to $(\theta_1-\theta_2)(\theta_1-\theta_3)(\theta_2-\theta_3)(\theta_1+\theta_2+\theta_3)=0$,
i.e., Theorem~\ref{vanDamProp41} is contained as a special case. Similar as in the proof of Theorem~\ref{CWnew}, we can use the formulas for the eigenvalues in
Theorem~\ref{thm_eigenvalues} and Corollary~\ref{extcor} to transfer \cite[Proposition 4.3]{DO} to $s$-sum sets:
{\theorem \label{CWnew2} Assume that $C(\Omega)^\bot$ is of length $n$ and has three {nonzero} weights $w_1<w_2<w_3.$ Let $b\ge 0$ and $s\ge 3$ be integers.
\begin{itemize}
\item $\Omega$ is an $s$-sum set iff $(w_3-w_2)(n(p-1)p^{m-2}-pw_1)^s+(w_1-w_3)(n(p-1)p^{m-2}-pw_2)^s+(w_2-w_1)(n(p-1)p^{m-2}-pw_3)^s=0$.
\item $\Omega\bigcup 0$ is an $s$-sum set iff $(w_3-w_2)(1+n(p-1)p^{m-2}-pw_1)^s+(w_1-w_3)(1+n(p-1)p^{m-2}-pw_2)^s+(w_2-w_1)(1+n(p-1)p^{m-2}-pw_2)^s=0$.
\item $\Gamma_{C(\Omega)}^b$ is an $s$-SWRG iff $(w_3-w_2)(b+n(p-1)p^{m-2}-pw_1)^s+(w_1-w_3)(b+n(p-1)p^{m-2}-pw_2)^s+(w_2-w_1)(b+n(p-1)p^{m-2}-pw_2)^s=0$.
\end{itemize}
}
%% \begin{proof} Follows by \cite{DO} Proposition 4.3, upon using Theorem 1 or its corollary.
%% \end{proof}

{\bf Example 4:} If $n=12$, $p=m=2$ the only triples $(w_1,w_2,w_3)$ satisfying the first condition of Theorem~\ref{CWnew2} %%the above theorem
for $s=5$ are %%those satisfying the conditions of Theorem~\ref{new}, namely
$\{11-i,12,13+i\},$ for $i=0,1,\dots,10$.

Note that $w_2=12=n$ in all cases. This can be partially explained by the fact that Equation~(\ref{eq_diophantine_vandam}) admits the solution $\theta_2=0$,
$\theta_1=-\theta_3$ for all odd $s$. In other words, a graph $\Gamma$ with these eigenvalues is an $s$-SWRG for all odd $s>1$, cf.~\cite[Proposition 4.2]{DO}. Using
Theorem~\ref{equiv} and Theorem~\ref{thm_eigenvalues} we can transfer this observation to a result that is new in the context of $s$-sum sets:
{\theorem \label{new} Assume that $C(\Omega)^\bot$ is of length $n$ and has three {nonzero} weights $w_1<w_2=n(p-1)p^{m-2}<w_3,$  with $w_1+w_3=2n(p-1)p^{m-2}.$
Then $\Omega$ is an $s$-sum set for all odd $s>1.$}
%% \begin{proof} Follows by \cite{DO} Proposition 4.2, upon using Theorem 1.
%% \end{proof}

In \cite[Theorem 4.4]{DO} it is essentially proven that if Equation~(\ref{eq_diophantine_vandam}) admits a solution $(\theta_1,\theta_2,\theta_3)$
with $\theta_1>\theta_2>\theta_3$ for two different values of $s>1$, then $\theta_2=0$ and $\theta_1=-\theta_3$. This directly leads to:
{\theorem\label{thm_w2} Assume that $C(\Omega)^\bot$ is of length $n$ and has three {nonzero} weights $w_1<w_2<w_3.$ If $w_2 \neq n(p-1)p^{m-2},$ or $w_1+w_3\neq 2n(p-1)p^{m-2},$ then there is at most one $s>1$ such that $\Omega$
is an $s$-sum set. }
%% \begin{proof} Follows by \cite[ Theorem 4.4]{DO}, upon using Theorem 1 or its corollary.
%% \end{proof}

If we are not in the situation of Theorem~\ref{new}, then any given code can yield an $s$-sum set for at most one value of $s$. There are other solutions
of Equation~(\ref{eq_diophantine_vandam}), e.g.\ $(5,\tfrac{1}{2}(-3+\sqrt{281}),\tfrac{1}{2}(-3-\sqrt{281}))$ for $s=5$. However, the found solutions are non-integral
and no graph that is an $s$-SWRG and attains these eigenvalues is known, see  \cite[p. 808]{DO}.
Since the eigenvalues corresponding to
an $s$-sum set are integral, we state:

\begin{conjecture}
\label{conj_s_sum_set}
$\Omega$ is an $s$-sum set for some $s>3$ iff $$(w_1,w_2,w_3)=(w_1,n(p-1)p^{m-2},2n(p-1)p^{m-2}-w_1),$$ for some $0<w_1<n(p-1)p^{m-2}$.
\end{conjecture}

At least for the special case $p=m=2$, i.e., $\Z_4$-codes, Conjecture~\ref{conj_s_sum_set} seems to be true also for $s=3$, see Conjecture~\ref{conj_w2_n} in
Section~\ref{sec_z4_codes}. In \cite{Co92} it has been shown that all $s$-sum sets of a special type also have to be $3$-sum sets.

%%%%%%%%%%%%%%%%%%%%%%%%%%%%%%%%%%%%55
\section{$\Z_4$-codes}
\label{sec_z4_codes}
Here we consider $\mathbb{Z}_4$ linear codes and assume that the reader is familiar with \cite{HKCSS}. When we say that $C$ is an $[n,k]$ code, where $k=(k_1,k_2)$ is the \emph{shape},
this means that $C$ is a linear code over $\mathbb{Z}_4$ of effective length $n$, which admits a generator matrix of the form
\begin{equation}
  \label{eq_generatormatrix}
  G=\begin{pmatrix}
    I_{k_1} & A & B\\
    0 & 2I_{k_2} & 2D
  \end{pmatrix},
\end{equation}
where $I_{k_1}$ and $I_{k_2}$ denote the $k_1\times k_1$ and $k_2\times k_2$ identity matrices, respectively, $A$ and $D$ are $0,1$-matrices
($\mathbb{Z}_2$-matrices), and $B$ is a $\mathbb{Z}_4$ matrix. We assume that the stated generator matrix consists of $n$ columns (and $k_1+k_2$ rows),
so that the condition effective length $n$ is equivalent to the property that the generator matrix does not contain an all-zero vector $\mathbf{0}$ as a column.
%%  (zero column for short).
The code size, i.e., the number of codewords, is given by $\# C=4^{k_1}\cdot 2^{k_2}= 2^{2k_1+k_2}$.
%% Here we consider the Lee weight $w_L$, i.e., the symbol $2$ is mapped to weight $2$, the symbol $0$ is mapped to weight $0$, and the symbols $1$ and $3$ are mapped to $1$.
%% In a more general setting, the
%% Lee weight is also called the homogeneous weight.  If the nonzero weights of $C$ are contained in $\left\{w_1,\dots,w_l\right\}$ we also speak of an
%% $\left[n,k,\left\{w_1,\dots,w_l\right\}\right]$ code. If all $l$ weights indeed occur, then we speak of an $l$-weight code (if some weights can be
%% missing, then we speak of an $\le l$-weight code). The support $\supp(c)$ of a codeword $c\in\mathbb{Z}_4^n$ is the set of nonzero coordinates, i.e.,
%% $\{1\le i\le n\,:\, c_i\neq 0\}$.
For $\Z_4$-codes we note that the homogeneous weight coincides with the Lee weight.

We  note the following compatibility between a $\Z_4$-code and its Gray image.
{\proposition\label{prop_gray_z4}If $C$ is a $\Z_4$-code with three Lee weights satisfying the hypotheses of Theorem \ref{CWnew}, then its binary image $C^G$ satisfies
the hypotheses of Theorem \ref{CWnew}.}
\begin{proof} As is well-known, the Hamming weights of $C^G$ are the Lee weights of $C$. If $n$ is the length of $C$ then $C^G$
is of length $2n$. The third condition of Theorem \ref{CWnew} becomes with $p=m=2$, $$w_1+w_2+w_3=3\frac{b+n(p-1)p^{m-1}}{p}=3\frac{b+2n}{2},$$
which is the same with $p=2, m=1$ and length $2n$.
\end{proof}

By $A_i$ we denote the number of codewords of $C$ of homogeneous weight $i$ and by $B_i$ we denote the number of codewords of homogeneous weight
$i$ of its dual code $C^\bot$. The cardinality of the dual code is $\# C^\bot=2^{2n-2k_1-k_2}$. Since we assume that
a corresponding generator matrix does not have a zero column, we have $B_1=0$ in our context.
%% Note that for $\Z_4$ the only possibility for a column of a generator matrix that is not regular is a column consisting of $0$s and $2$s,
%% which permits a dual codeword of weight $2$. Thus
As mentioned at the end of Subsection~\ref{subsec_homogeneous weight}, every $\Z_4$ code has minimum dual distance $d^\bot\ge 3$ iff it
is regular and projective. The numbers $A_i$ can be encoded in a weight enumerator of $C$:
\begin{equation}
  \operatorname{\mathrm{Hom}}_C(X,Y)=\sum_{c\in C} X^{2n-\whom(c)}Y^{\whom(c)}=\sum_{i=0}^{2n} A_iX^{2n-i}Y^i.
\end{equation}
%%Obviously, we have $A_0=1$.
The corresponding MacWilliams identity %%for this the Lee weight enumerator
is given by $\operatorname{\mathrm{Hom}}_C(X,Y) =\tfrac{1}{\#C^\bot } \cdot\operatorname{\mathrm{Hom}}_{C^\bot}(X+Y,X-Y)$, see e.g.\cite{HKCSS}.
Note that this formula equals the Hamming weight enumerator for a linear code over $\mathbb{F}_2$ with cardinality $\# C=2^k$ if we replace $n$ by $2n$ and $k$ by
$2k_1+k_2$. As mentioned in the introduction and used in the proof of Lemma~\ref{prop_gray_z4}, the same conclusion can be obtained via the Gray map.
Thus, we can easily rewrite the classical \emph{MacWilliams identities} \cite{macwilliams1977theory} or the \emph{Pless power moments}
\cite{pless1963power} for the Hamming weight over $\F_2$. Using $A_0=B_0=1$ and $B_1=0$ the first four \emph{(Pless) power moments}  for the $\whom$
are given by:
%% So, we can e.g.\ rewrite the classical \emph{MacWilliams identities} for the Hamming weight over $\F_2$, see e.g.~\cite{macwilliams1977theory}, to
%% \begin{equation}
%%   \label{mac_williams_identies}
%%   \sum_{j=0}^{2n-\nu} {{2n-j}\choose\nu} A_j=q^{2k_1+k_2-\nu}\cdot \sum_{j=0}^\nu {{2n-j}\choose{2n-\nu}}B_j\quad\text{for }0\le\nu\le 2n,
%% \end{equation}
%% where, additionally, $A_0=B_0=1$. The fact that the $B_i$ are uniquely determined by the $A_i$ can e.g.\ be seen by providing explicit equations for each $B_i$ in
%% dependence of the $A_j$. Those formulas involve the so-called \emph{Krawtchouk polynomials} \cite{krawtchouk1929generalisation}. Assuming that $n$ is the effective
%% length of the code $C$ is equivalent to $B_1=0$. The first four MacWilliams (for the Lee weight over $\mathbb{Z}_4$ identities can be rewritten to:
\begin{eqnarray}
  \sum_{i>0} A_i &=& 2^{2k_1+k_2}-1,\label{eq_ppm1}\\
  \sum_{i\ge 0} iA_i &=& 2^{2k_1+k_2}n,\label{eq_ppm2}\\
  \sum_{i\ge 0} i^2A_i &=& 2^{2k_1+k_2-1}(B_2+n(2n+1)),\label{eq_ppm3}\\
  \sum_{i\ge 0} i^3A_i &=& 2^{2k_1+k_2-2}(3(B_2n-B_3)+2n^2(2n+3))\label{eq_ppm4}.
\end{eqnarray}
%% In this special form they are the analogue of the first four \emph{(Pless) power moments} (for the Hamming weight over $\F_2$), see \cite{pless1963power}.
We remark, that the first three equations (including $B_1$) can also be found in \cite[Theorem 3.1]{SW}.

In this section we are interested in structural results for $l$-weight codes where $l$ is small. Mostly we will restrict our attention to projective codes only. In some
cases we will require additional constraints on the sum $\sum_{i=1}^l w_i$ of weights (that are motivated by strongly walk regular graphs, see e.g.\ \cite{DO}
and see also \cite{SS}, which studies the field case).

\subsection{Theoretical results for $\Z_4$-codes with few weights}
In this section we give only very few results that are of interest for the remaining part of this paper. For more information we refer the interested reader
to e.g.~\cite{SW,ShiXuYang2017}, where one-Lee weight codes and projective two-Lee weight codes are studied. For every pair $(k_1,k_2)$ of non-negative
integers there exists exactly one non-isomorphic $1$-weight $[n,(k_1,k_2)]$ code (for a unique length $n$ depending on $k_1$ and $k_2$), see \cite{C}.

\begin{lemma}\
  \label{lemma_even_weight_subcode}
  Let $C$ be an $[n,(k_1,k_2)]$ code. By $C_2$ we denote the subcode of $C$ spanned be the codewords of even weight. The cardinality of $C_2$ is either $2^{2k_1+k_2-1}$ or
  $2^{2k_1+k_2}$ and all codewords of $C_2$ have an even weight.
\end{lemma}
\begin{proof}
  Let $E$ be the code consisting of codewords of $C$ of even Lee weight. This code is linear as the dual of the code spanned by the all-$2$ vector, so
  that $C_2=E\cap C$. Thus $C_2$ is linear and its size depends on whether $2\cdot {\bf 1}\in C^\bot$ or not.
\end{proof}
We also call $C_2$ the even-weight subcode of $C$.

\begin{lemma}
  \label{lemma_3wt_odd}
  Let $C$ be an $[n,(k_1,k_2)]$ code with $t\ge 2$ different nonzero weights. Then, at most $t-1$ weights can be odd.
\end{lemma}
\begin{proof}
  Let $c\in C$ be a codeword of odd weight, then $2c\in C$ is a codeword of even weight.
\end{proof}

\begin{lemma}
  \label{lemma_3wt_parametric}
  Let $C$ be an %%projective
  $[n,(k_1,k_2)]$ three-weight code with $d^\bot\ge 3$ and weights $w_1$, $w_2$, and $w_3$ that occur $A_1$, $A_2$,
  and $A_3$ times, respectively. Using the abbreviation $y=2^{2k_1+k_2-1}$ we have
  \begin{eqnarray}
    A_1 &=& \frac{y\cdot\left(2n^2 - 2nw_2 - 2nw_3 + 2w_2w_3 + n\right) - w_2w_3}{(w_2-w_1)(w_3-w_1)} \\
    A_2 &=& \frac{y\cdot\left(2n^2 - 2nw_1 - 2nw_3 + 2w_1w_3 + n\right) - w_1w_3}{(w_2-w_3)(w_2-w_1)} \label{eq_A2}\\
    A_3 &=& \frac{y\cdot\left(2n^2 - 2nw_1 - 2nw_2 + 2w_1w_2 + n\right) - w_1w_2}{(w_3-w_1)(w_3-w_2)}
  \end{eqnarray}
  and
  \begin{eqnarray}
    3B_3 &=& 2n^2(2n+3)-\left(w_1+w_2+w_3\right)2n(2n+1)-4w_1w_2w_3\nonumber\\
        & & +2\left(w_1w_2+w_1w_3+w_2w_3\right)2n+\frac{2w_1w_2w_3}{y}\nonumber
  \end{eqnarray}
  for the number of dual codewords of weight $3$.
\end{lemma}
\begin{proof}
  Solving the first three power moment equations (\ref{eq_ppm1})-(\ref{eq_ppm3}) for $A_1$, $A_2$, and $A_3$ gives the first statement. Plugging into the fourth
  power moment equation (\ref{eq_ppm4}) and solving for $3B_3$ gives the last equation of the statement.
\end{proof}
We remark that the integrality of $B_3$ implies that the product of the three weights $w_1w_2w_3$ is divisible by $y/2=2^{2k_1-k_2-2}$.
Now let us look at the first case of Theorem~\ref{CWnew}, which is $w_1+w_2+w_3=3n$ in our situation. Our first observation is
that the length $n$ has to be even.

\begin{lemma}
  \label{lemma_n_divisible_by_2}
  Let $C$ be an %%projective
  $[n,(k_1,k_2)]$ $3$-weight code with $d^\bot\ge 3$ and weights $w_1$, $w_2$, and $w_3$ satisfying $w_1+w_2+w_3=3n$. Then, $n\equiv 0\pmod 2$.
\end{lemma}
\begin{proof}
  Assume that $n$ is odd. Then, also $w_1+w_2+w_3=3n$ is odd. Using Lemma~\ref{lemma_3wt_odd} we conclude that exactly one weight -- say $w_2$ -- is odd.
  We consider the even weight subcode $C_2$ of $C$. Due to Lemma~\ref{lemma_even_weight_subcode} $C_2$ only contains codewords of even weight and has
  cardinality $\#C_2=2^{2k_1+k_2-1}$, i.e., $A_2=2^{2k_1+k_2-1}$. Since $C$ is projective $C_2$ has effective length $n'\in\{n-1,n\}$.
  Equation~(\ref{eq_ppm2}) gives
  $$
    w_1A_1+w_2A_2+w_3A_3=2^{2k_1+k_2}n
  $$
  for $C$ and
  $$
    w_1A_1+w_3A_3=2^{2k_1+k_2-1}n'
  $$
  for $C_2$. Taking the difference and dividing by $A_2=2^{2k_1+k_2-1}$ yields $w_2=n+\left(n-n'\right)$. Since $w_2,n$ are odd and $n-n'\in\{0,1\}$, we
  have $n'=n$ and $w_2=n$. So, let us write $w_1=n-t$ and $w_3=n+t$ for some odd integer $t$ (observe that $w_1$ and $w_3$ have to be even).
  With this, Equation~(\ref{eq_A2}) and $A_2=2^{2k_1+k_2-1}=y$ gives
  $(y-1)t^2=(y-n)n$.
  Note that $t=1$ gives $y=n+1$ or $n=1$. In the first case we have $A_1=0$ and $w_1=0$ in the second, so that we can assume $t\ge 3$.
  Solving the above equations for $n$ gives
  $$
    n=\frac{y}{2}\pm \frac{\sqrt{-4t^2y + 4t^2 + y^2}}{2}.
  $$
  So, we consider the diophantine equation
  \begin{equation}
    x^2=-4t^2y + 4t^2 + y^2\label{eq_diophant_1},
  \end{equation}
  where $y$ is a power of $2$ and all integers are positive.\footnote{Note that Equation~(\ref{eq_diophant_1}) has infinitely many
  solutions for $t\in\{-1,0,1\}$. Also if $y$ is not a power of $2$ there are lots of solutions, e.g.
  $
    (x,t,y)\in\left\{(1,2,15),(11,2,21),(14,3,40),(17,6,145)\right\}
  $.}
  In our situation we have $t\ge 3$, so that $y\ge 64=2^6$, since $x^2$ would be negative for $t\ge 3$ and $y\le 32$. Now let us set $x=2z$ and $y=2^{r+1}$, i.e.,
  $r\ge 5$. With this we can rewrite Equation~(\ref{eq_diophant_1}) to
  \begin{equation}
    (z-t)(z+t)=2^{2r}-2^{r+1}t^2\label{eq_diophant_2}.
  \end{equation}
  Since $t$ is odd, $z$ has to be odd too. We conclude that either $z-t$ or $z+t$ has to be divisible by $2^r$ from the fact that
  the greatest common divisor of $z-t$ and $z+t$ divides $2t$ and $r\ge 5$.

  Assume that $z=s2^r-t$ for some nonzero integer $s$ (this is the case that $2^r$ divides z+t). Plugging in into Equation~(\ref{eq_diophant_2}), dividing
  by $2^{r+1}$, and simplifying gives $\left(2^{r-1}s-t\right)\cdot s=2^{r-1}-t^2$, so that
  $$
    2^{r-1}\cdot\left(s^2-1\right) = t\cdot (s-t).
  $$
  Note that $s^2\ge 1$. If $s^2=1$, then $s=t$ or $t=0$, which contradicts $t\ge 3$. Thus, the left hand side is positive and $s>t\ge 3$.
  This implies $t(s-t)\le s(s-t)<2^{r-1}(s^2-1)$ -- contradiction.
  %% Since the right hand side is strictly less than $s^2-1$ we obtain a contradiction.

  Assume that $z=s2^r+t$ for some nonzero integer $s$ (this is the case that $2^r$ divides z-t). Plugging in into Equation~(\ref{eq_diophant_2}), dividing by
  $2^{r+1}$, and simplifying gives $\left(2^{r-1}s+t\right)\cdot s=2^{r-1}-t^2$, so that
  $$
    2^{r-1}\cdot\left(s^2-1\right) = t\cdot (-s-t).
  $$
  Note that $s^2\ge 1$. If $s^2=1$, then $s=-t$ or $t=0$, which contradicts $t\ge 3$. Thus, the left hand side is positive and $-s>t\ge 3$.
  This implies $t(s-t)\le s(s-t)<2^{r-1}(s^2-1)$ -- contradiction.
  %% Since the right hand side is strictly less than $s^2-1$ we obtain a contradiction.
\end{proof}

We remark that in the case where also $C_2$ is projective the proof can be shortened significantly since \cite[Corollary 16]{byrne2012properties}
implies that $2t$ is a power of two. An example where Lemma~\ref{lemma_n_divisible_by_2} excludes parameters of a code is given by $n=29$, $2k_1+k_2=8$,
$w=(24,31,32)$, and weight distribution $0^1 24^{76} 31^{128} 32^{51}$. Applying the MacWilliams transform gives $B_0=1$, $B_1=B_2=0$, $B_3=164$, and indeed all
$B_i$ are non-negative integers. Moreover, the number of even-weight codewords is $128$, i.e., half the size of the code.

Next we show that under the assumption of Theorem~\ref{thm_w2} the length and the three weights all are divisible by large powers of $2$
if the length is suitably large.
\begin{proposition}
  \label{prop_divisibility_n}
  Let $C$ be an %%projective
  $[n,(k_1,k_2)]$ $3$-weight code with $d^\bot\ge 3$ and weights satisfying $w_1+w_2+w_3=3n$ and $w_2=n$. For each positive integer $r$ there exists an integer $N(r)$ such
  that $n\ge N(r)$ implies that $2^r$ divides $n$ and $2^{r-1}$ divides all three weights $w_1,w_2,w_3$.
\end{proposition}
\begin{proof}
  Let $t$ be a positive integer with $w_1=n-t$ and $w_3=n+t$. With this (and $y=2^{2k_1+k_2-1}$) the equations of Lemma~\ref{lemma_3wt_parametric} are equivalent to
  \begin{eqnarray}
    A_1 &=& \frac{n(y - n - t)}{2t^2} \\
    A_2 &=& %%\frac{(2y - 1)(2t^2-n) +n(2n-1)}{2t^2} \\
            \frac{t^2(2y-1)-n(y-n)}{t^2}\label{eq_A2_w_2_n}\\
    A_3 &=& \frac{n(y - n +t)}{2t^2}
  \end{eqnarray}
  and
  \begin{eqnarray}
    3B_3&=&  \frac{2n(n-t)(n+t)}{y}.
  \end{eqnarray}
  Since $A_3-A_1=\tfrac{n}{t}$ the effective length $n$ has to be divisible by $t$. From $A_2\in\mathbb{N}$ we conclude that $t^2$ divides $n(n-y)$. So, if
  $p^l$ divides $t$ for some odd prime $p$, then $p^{2l}$ has to divide $n$ since $y$ is a power of $2$. As an abbreviation we set $k=2k_1+k_2$. Now let us
  try to parameterize $t=2^u\cdot v$ and $n=2^x\cdot v^2\cdot z$ for odd positive integers $v,z$ and non-negative integers $u$, $x$. Plugging in and simplifying gives
\begin{eqnarray}
    A_1 &=& \frac{z\cdot (2^{k-u-1}-2^{x-u} v^2z -v)}{2^{u+1-x}} \label{eqs1}\\
    A_2 &=& 2^{2(x-u)}v^2z^2 + 2^{k} - 2^{x+k-2u-1}z - 1 \label{eqs2}\\
    A_3 &=& \frac{z\cdot (2^{k-u-1}-2^{x-u}v^2z +v)}{2^{u+1-x}} \label{eqs3}\\
    3B_3&=&  \frac{v^4z\cdot \left(2^{x-u}vz - 1\right)\cdot\left(2^{x-u}vz + 1\right)}{2^{k-x-2u-2}}\label{eqs4},
\end{eqnarray}
where $u\le x$ ($t$ divides $n$) and $x\le k-1$ ($n\le 2^k-1$ since $C$ is projective).

If $k-x-2u-2\ge 1$ then $B_3\in \mathbb{N}$ and $v,z\equiv 1\pmod 2$ imply $u=x$. Since $\gcd(vz-1,vz+1)=2$, we have that $2^{k-3u-3}$ either divides
$vz-1$ or $vz+1$. So, we use the parameterization $vz=s\cdot \left(2^{k-3u-3}\right)+\alpha$ for some positive integer $s$ and $\alpha\in\{-1,1\}$. With this $A_1>0$ gives
$$
2^{k-u-1}-v\left(s\cdot \left(2^{k-3u-3}\right)+\alpha+1\right)>0,
$$
so that $vs<2^{2u+2}$, i.e., $sv\le 2^{2u+2}-1$. Now $A_2>0$ gives $v^2z^2 + 2^{k} > 2^{k-u-1}z$, which is equivalent to
\begin{equation}
  sv(vz)^2 + 2^{k}sv > s2^{k-u-1}vz= s^22^{2k-4u-4}+\alpha s2^{k-u-1}>s^22^{2k-4u-4}-s2^{k+2u+2}.
\end{equation}
Since $sv\le 2^{2u+2}-1$ the left hand side is at most
\begin{eqnarray*}
  && s^2 2^{2k-4u-4} - s^2 2^{2k-6u-6} + s\alpha 2^{k-u} - s\alpha 2^{k-3u-2} \\
  &&+ \alpha^2 2^{2u+2} - \alpha^2 + 2^{k+2u+2} - 2^k\\
  &\le& s^2 2^{2k-4u-4} -s^2 2^{2k-6u-6}  +2s 2^{k-u} + 2^{2u+2} +2^{k+2u+2} \\
  &\le & s^2 2^{2k-4u-4} -s^2 2^{2k-6u-6}+3s 2^{k+2u+2}.
\end{eqnarray*}
Thus
\begin{equation}
  4s\cdot2^{k+2u+2} > s^22^{2k-6u-6}
\end{equation}
has to be satisfied, so that $k\le 8u+9$ and
\begin{equation}
  x=u\ge \frac{k-9}{8}.
\end{equation}

 Otherwise we have $k\le x+2u+2$. From $A_1>0$ we conclude $y-n>0$. Since $y-n$ is an integer and both $y$ and $n$ are divisible by $2^x$ we have $y-n\ge 2^x$. Now
 $A_2>0$, Equation~(\ref{eq_A2_w_2_n}), and $2y-1\le 2^k$ imply
 $$
   v^22^{2u}\cdot 2^k-2^xv^2z\cdot 2^x>0,
 $$
 %%so that $k+2u+1>2x$, i.e.,
 so that $k+2u>2x$, i.e.,
 \begin{equation}
   %%k\ge 2x-2u.
   k\ge 2x-2u+1.
 \end{equation}
 Combined with $k\le  x+2u+2$ we obtain
 %%$x\le 4u+2$
 $x\le 4u+1$ and
 %%$k\le 6u+4$, i.e.,
 $k\le 6u+3$, i.e.,
 \begin{equation}
   %%x\ge u\ge \frac{k-4}{6}.
   x\ge u\ge \frac{k-3}{6}.
 \end{equation}

 In both cases we can conclude $x\ge u\ge \tfrac{k-9}{8}$, so that the result follows from $n<2^k$.
\end{proof}

\begin{conjecture}
  \label{conj_w2_n}
  Let $C$ be an $[n,(k_1,k_2)]$ $3$-weight code with $d^\bot\ge 3$ and weights satisfying $w_1+w_2+w_3=3n$ and $w_1<w_2<w_3$. Then, $w_2=n$.
\end{conjecture}
We remark that in order to prove Conjecture~\ref{conj_w2_n} the Diophantine conditions of Lemma~\ref{lemma_3wt_parametric}, i.e., $A_1,A_2,A_3\in\mathbb{N}_{>0}$ and
$B_3\in\mathbb{N}_{\ge 0}$ are not sufficient. Up to $n=50$ we have the following exceptional tuples $\left(n,w_1,w_2,w_3,y,A_1,A_2,A_3,B_3\right)$:
\begin{itemize}
  \item $(29,24, 31, 32, 64, 76, 128, 51, 164)$
  \item $(33, 29, 32, 38, 64, 64, 111, 80, 157)$
  \item $(34, 30, 32, 40, 128, 64, 299, 148, 36)$
  \item $(50, 46, 48, 56, 64, 32, 145, 78, 580)$
\end{itemize}
So, from Lemma~\ref{lemma_n_divisible_by_2} we can conclude that Conjecture~\ref{conj_w2_n} is true for all $n<34$.
We remark that the examples for $n\in\{29, 34, 50\}$ satisfy $B_i\in\mathbb{N}_{\ge 0}$ for all $1\le i\le n$, i.e., the MacWilliams identities
are not sufficient in order to conclude the non-existence of the corresponding code.

\subsection{Classification of $3$-weights $\Z_4$-codes with $d^\bot\ge 3$ and short length}
\label{subsec_z4_codes_small_length}

Having Theorem~\ref{CWnew} in mind we aim to classify $3$-weights $\Z_4$-codes with $d^\bot\ge 3$ and short length.
Optimal linear codes over $\mathbb{Z}_4$ with length $n\le 7$ were classified in \cite{wong2002clasification}.
Without the restriction to optimal codes, all linear codes over $\mathbb{Z}_4$ with length $n\le 10$ and $2K_1+k_2\le 8$ where classified in \cite{feulner2011canonization} for the homogeneous weight. The
corresponding tables at \url{http://www.algorithm.uni-bayreuth.de/en/research/Coding_Theory/CanonicalForm/Classification/index.html} contain the complete counts for $n\le 10$. E.g.\ for
$n=10$ and $(k_1,k_2)=10$ there are exactly $8\,848\,026$ non-isomorphic codes.
For each given $n$ our approach is to loop over all $1\le w_1<w_2<w_3\le 2n$ such that $A_1,A_2,A_3,B_3$, as specified in Lemma~\ref{lemma_3wt_parametric},
are integers with $A_1,A_2,A_3\ge 1$ and $B_3\ge 0$. Since $2^{2k_1+k_2-2}$ has to divide $w_1w_2w_3$, which is due to $B_3\in\mathbb{N}$, we obtain a finite list of
possibilities for each length $n$. %% 3wt_rational.cpp, where instead of n we have to enter 2n
Note that $d^\bot \ge 3$ implies $k_1\ge 1$. With respect to Theorem~\ref{CWnew} we consider all remaining tuples where $w_1+w_2+w_3\ge 3n$ and $w_1+w_2+w_3\equiv 0\pmod 3$.
In Table~\ref{tab_attainable} we list the parameters that can be attained by a $\mathbb{Z}_4$ codes and refer to a corresponding generator matrix. As abbreviations we use
$w=(w_1,w_2,w_3)$, $S=w_1+w_2+w_3$, and $A=(A_1,A_2,A_3)$. The parameters of computationally excluded parameters are listed in Table~\ref{tab_excl_eq} and Table~\ref{tab_excl_lt}.
For the latter we perform an exhaustive search using the given restrictions on the weights $w_i$, the length $n$, and the shape $(k_1,k_2)$. %% z4_tss.cpp

\begin{table}
  \begin{center}
    \begin{tabular}{lllll}
      \hline
      $n$ & $w$ & $A$ & $(k_1,k_2)\to G$ \\
      \hline
      6 & (4,6,8) & (6,16,9) & $(2,1)\to G_{6,1}^1$\\
      6 & (4,6,8) & (18,24,21) & $(3,0)\to G_{6,2}^1$, $(2,2)\to G_{6,3}^1$\\
      8 & (4,8,12) & (1,27,3) & $(2,1)\to G_{8,1}^1$ \\
      8 & (4,8,12) & (5,51,7) & $(3,0)\to G_{8,2}^1$, $(2,2)\to G_{8,3}^1$ \\
      \hline
      3 & (2,4,6) & (15,15,1) & $(2,1)\to G_{3,1}^2$\\
      5 & (4,6,8) & (16,12,3) & $(2,1)\to G_{5,1}^2$\\
      7 & (6,8,10) & (42,7,14) & $(3,0)\to G_{7,1}^2$\\
      9 & (8,10,12) & (15,12,4) & $(2,1)\to G_{9,1}^2$\\
      10 & (8,12,16) & (62,64,1) & $(3,1)\to G_{10,1}^2$ \\
      10 & (8,12,16) & (130,120,5) & $(4,0)\to G_{10,2}^2$ \\
      \hline
    \end{tabular}
    \caption{Attainable parameters for $S\equiv 0\pmod 3$}
    \label{tab_attainable}
  \end{center}
\end{table}

$$
  G_{6,1}^1=\begin{pmatrix}
  1 0 1 1 1 2\\
  0 1 0 3 3 1\\
  0 0 2 2 0 0
  \end{pmatrix}\!\!,
  G_{6,2}^1=\begin{pmatrix}
  1 0 0 1 2 2\\
  0 1 0 2 1 2\\
  0 0 1 3 3 1
  \end{pmatrix}\!\!,
  G_{6,3}^1=\begin{pmatrix}
  1 0 1 1 1 2\\
  0 1 1 1 2 1\\
  0 0 2 0 0 2\\
  0 0 0 2 2 0
 \end{pmatrix}
$$

$$
  G_{8,1}^1=\begin{pmatrix}
  1 0 1 0 1 1 2 2\\
  0 1 1 1 2 3 1 1\\
  0 0 2 2 0 2 0 2
  \end{pmatrix}\!\!,
  G_{8,2}^1=\begin{pmatrix}
  1 0 0 0 1 2 2 2\\
  0 1 0 2 2 1 1 1\\
  0 0 1 1 0 0 1 3
  \end{pmatrix}\!\!,
  G_{8,3}^1=\begin{pmatrix}
  1 0 1 1 1 1 1 2\\
  0 1 1 1 2 3 3 1\\
  0 0 2 0 0 0 2 0\\
  0 0 0 2 0 0 2 0
 \end{pmatrix}
$$

$$
  G_{3,1}^2=\begin{pmatrix}
  1 0 1\\
  0 1 1\\
  0 0 2\\
  \end{pmatrix}\!\!,
  G_{5,1}^2=\begin{pmatrix}
  1 0 1 2 2\\
  0 1 1 1 1\\
  0 0 2 0 2
  \end{pmatrix}\!\!,
  G_{7,1}^2=\begin{pmatrix}
  1 0 0 1 1 1 2\\
  0 1 0 1 2 3 1\\
  0 0 1 2 1 3 3
  \end{pmatrix}
$$
$$
  G_{9,1}^2=\begin{pmatrix}
  1 0 1 1 1 1 1 2 2\\
  0 1 1 2 2 3 3 1 1\\
  0 0 2 0 2 0 2 0 2
  \end{pmatrix}\!\!,
  G_{10,1}^2=\begin{pmatrix}
  1 0 0 0 2 1 1 1 1 1\\
  0 1 0 1 2 2 3 0 1 0\\
  0 0 1 0 1 1 1 2 3 3\\
  0 0 0 2 2 2 0 2 2 2
  \end{pmatrix}\!\!,
  G_{10,2}^2=
  \begin{pmatrix}
  1 0 0 0 2 1 2 1 1 0\\
  0 1 0 0 0 3 1 2 3 2\\
  0 0 1 0 1 3 2 0 1 2\\
  0 0 0 1 2 1 0 2 3 1
  \end{pmatrix}
$$

\begin{table}
  \begin{center}
    \begin{tabular}{lllll}
      \hline
      $n$ & $2k_1+k_2$ & $w$ & $A$ & $(k_1,k_2)$ \\
      \hline
      2 & 3 & (1,2,3) & (1,3,3) & (1,1)\\
      4 & 4 & (2,4,6) & (1,11,3) & (2,0), (1,2)\\
      4 & 5 & (2,4,6) & (5,19,7) & (2,1), (1,3) \\
      4 & 6 & (2,4,6) & (13,35,15) & (3,0), (2,2), (1,4) \\
      6 & 5 & (4,6,8) & (6,16,9) & (1,3) \\
      6 & 6 & (4,6,8) & (18,24,21) & (1,4) \\
      8 & 5 & (6,8,10) & (6,15,10) & (2,1), (1,3) \\
      8 & 6 & (6,8,10) & (22,15,26) & (3,0), (2,2), (1,4) \\
      8 & 7 & (6,8,10) & (54,15,58) & (3,1), (2,3), (1,5) \\
      8 & 5 & (4,8,12) & (1,27,3) & (1,3) \\
      8 & 6 & (4,8,12) & (5,51,7) & (1,4) \\
      8 & 7 & (4,8,12) & (13,99,15) & (3,1), (2,3), (1,5)\\
      10 & 5 & (8,10,12) & (5,16,10) & (2,1), (1,3) \\
      10 & 6 & (8,10,12) & (25,8,30) & (3,0), (2,2), (1,4) \\
      \hline
    \end{tabular}
    \caption{Computationally excluded cases for $S=3n$}
    \label{tab_excl_eq}
  \end{center}
\end{table}

\begin{table}
  \begin{center}
    \begin{tabular}{lllll}
      \hline
      $n$ & $2k_1+k_2$ & $w$ & $A$ & $(k_1,k_2)$ \\
      \hline
      3 & 5 & (2,4,6) & (15,15,1) & (1,3)\\
      5 & 5 & (4,6,8) & (16,12,3) & (1,3)\\
      7 & 5 & (6,8,10) & (16,11,4) & (2,1), (1,3)\\
      7 & 6 & (6,8,10) & (42,7,14) & (2,2), (1,4)\\
      7 & 7 & (4,8,12) & (31,95,1) & (3,1), (2,3), (1,5)\\
      7 & 8 & (4,8,12) & (65,187,3) & (4,0), (3,2), (2,4), (1,6) \\
      9 & 5 & (8,10,12) & (15,12,4) & (1,3) \\
      9 & 6 & (8,11,14) & (43,16,4) & (3,0), (2,2), (1,4)\\
      10 & 7 & (8,12,16) & (62,64,1) & (2,3), (1,5) \\
      10 & 8 & (8,12,16) & (130,120,5) & (3,2), (2,4), (1,6) \\
      \hline
    \end{tabular}
    \caption{Computationally excluded cases for $S>3n$, $S\equiv 0\pmod 3$}
    \label{tab_excl_lt}
  \end{center}
\end{table}

We remark that more examples can be found for the parameters:
\begin{itemize}
  \item $n=15$, $k=9$, $w=(12,16,20)$, $A=(190,255,66)$;
  \item $n=18$, $k=8$, $w=(16,20,24)$, $A=(153,72,30)$;
  \item $n=22$, $k=7$, $w=(20,24,28)$, $A=(71,43,13)$.
\end{itemize}

\subsection{Cyclic Kerdock code}
%We assume that the reader is familiar with \cite{HKCSS}.
The Kerdock code $K$ of length $2^{s}$ with $s\ge 3$ odd over $\Z_4$ can be expressed as
$$K=\Z_4 {\bf j} \oplus Q,$$ where $\bf j$ denotes the all-one vector. The code $Q$ is a free code of size $4^s$ with weight spectrum
$$2^s, 2^s\pm 2^{\frac{s-1}{2}}. $$ By puncturing on the first coordinate we obtain a cyclic code of length $n=2^{s}-1$ that  was used in \cite{S} to generate
low correlation quadriphase sequences. This code is called $K^-$ in \cite{HKCSS} and has the same weights as $K$. It can be shown to be projective by the same
type of arguments as in \cite[\S 3.3]{HKCSS}. The sum of these weights being $3\cdot 2^s=3(n +1)$, we see that, by Theorem \ref{CWnew} with $b=2$,
the syndrome graph of $K^-$ with a double loop at each vertex is a $3$-SWRG. It can be easily checked that the parameters of $K^-$ go in line
with our results the previous subsection

We observe that the binary images have a length congruent to $2 \pmod{4}$, which is enough to show they are not linear \cite{K+}. This is an alternative proof
that Kerdock codes cannot be linear. The classical proof builds on the fact that the Preparata code, or any code with the same weight distribution is nonlinear \cite{G}.
%%%%%%%%%%%%%%%%%%%%%%%%%%%%%%%%%%%%55
\section{Extension to chain rings}
\label{sec_cr}
In this section, we briefly indicate how the preceding work extends when replacing $\Z_p^m$ by a finite chain ring $R$ of depth $r\ge 2$, say, and residue field
$\F_q$, for some prime power $q$. Recall that such a ring is a local ring whose lattice of ideals form a chain of length $r$. In particular, the maximal ideal
is principal, generated by an element $\gamma,$ a nilpotent element of nilpotency index $r$ \cite[Chap. 3]{SAS}. An example is $\Z_p^m$, when $r=m$, and $q=p$.
The definition of the homogeneous weight becomes

\begin{equation*}
\whom(x) =
\begin{cases}
    0  &  \text{if $x = 0$},\\
    q^{r-1}   & \text{if $0\neq x\in (\gamma^{e-1})$},\\
    (q-1)q^{r-2}  & \text{otherwise}.
\end{cases}
\end{equation*}

The definition of the syndrome graph in Subsection~\ref{subsec_syndrome_graph} to this generalized situation only needs the obvious
modifications. Also the proof of \cite[Theorem 4.1]{S+} can be generalized to conclude the connection between coset graph eigenvalues and weights:
\begin{theorem}\label{10cr}
  Suppose that $C$ is a regular, projective linear code over $R$ with
  homogeneous weights $w_i$ and corresponding weight distribution
  $A_i=\abs{\{x\in C;\whom(x)=w_i\}}$. Then
  the eigenvalues of $\Gamma(C^\bot)$ are $n(q-1)q^{r-1}-qw_i$ with
  multiplicity $A_i$.
\end{theorem}

In particular, for the coset graph with loops, we have the following result.
{\corollary \label{extcorcr} If $C$ is an $R$-code, of minimum distance at least three, with dual weight distribution $[\langle i,A_i\rangle]$, then the spectrum of $\Gamma_C^b$ is
$\{(b+n(q-1)q^{r-1}-qi)^{A_i}\}$. Thus $A_i$ is the frequency
of the homogeneous weight $i$ in $C^\bot$ and the multiplicity of the eigenvalue $b+n(q-1)q^{r-1}-qi$.}

From there, upon observing that Theorem \ref{equiv} is valid for any finite commutative ring,
the main condition for the existence of a TSS becomes
{\theorem \label{CWnewcr} Assume that $C(\Omega)^\bot$ is of length $n$ and has three {nonzero} weights $w_1<w_2<w_3.$ Let $b$ be any integer.
\begin{itemize}
\item $\Omega$ is a TSS iff $w_1+w_2+w_3=3n(q-1)q^{r-2}$;
\item $\Omega \bigcup 0$ is a TSS iff $w_1+w_2+w_3=3\frac{1+n(q-1)q^{r-1}}{q}$;
\item $\Gamma_{C(\Omega)}^b$ is a $3$-SWRG iff $w_1+w_2+w_3=3\frac{b+n(q-1)q^{r-1}}{q}$;
\end{itemize}
}
We remark that $r\ge 2$ and $q$ not being divisible by $3$ the second case is impossible and $b$ has to be divisible by $q$ in the third case.

%%%%%%%%%%%%%%%%%%%%%%%%
\subsection{ Trace codes over chain  rings}
\label{subseb_tcr}
%Trace codes over rings can give few-weight codes over fields by application of a Gray map.
 Consider the chain ring $\F_p+u\F_p,$ with $u^2=0.$ This is a chain ring with $r=2$ and $q=p.$
The  codes $C(m,p)$ of length $\frac{p^{2m}-p^m}{2}$ with $m$ singly even in \cite[\S 5.1]{Sh3} are defined as
$$ \{(Tr(ax))_{x\in L} |\; a \in \mathcal R \},$$ where $L=Q+u\F_p^m,$ \, $Q$ is the set of squares of $\F_{p^m},$ and $\mathcal{R}=\F_{p^m}+u\F_{p^m}.$
Note that, since $m$ is even, we have $\F_p^\times \subseteq Q.$  This implies that these codes are replications by a factor $p-1$
  of projective codes $P(m,p)$ say of length $n=\frac{p^{2m}-p^m}{2(p-1)}$.
These latter codes have weights $$p^{2m-1}-p^{m-1},\,p^{2m-1}-p^{m-1}- p^{m-1}(p^{m/2}+1),\, p^{2m-1}-p^{m-1}+ p^{m-1}(p^{m/2}-1),$$ and so
  satisfy the relation $$ w_1+w_2+w_3=3 (1-1/p)n,$$ for the weight of the Gray map in
\cite{Sh3}, normalized by a factor $2$. This shows, by the first condition of Theorem \ref{CWnewcr}, that
the coset graph $\Gamma_C$  is a $3$-SWRG, or, equivalently, that the $\Omega$ such that $C(\Omega)=C$ is a TSS.

%%%%%%%%%%%%%%%%%%%%%%%%%%%%%%%%%%

\subsection{Generalized Teichm\"uller codes}
Let $q = 2^r$ and $R = \operatorname{GR}(4,r)$ the Galois ring of characteristic $4$ and residue field $\F_{2^r}$.
The order of $R$ is $q^2$.
For integers $k\geq 2$ and
\[
	s\in\begin{cases}
		\{0,2,4,\ldots,(k-1)r\} & \text{if }k\text{ odd;} \\
		\{r,r+2,r+4,\ldots,(k-1)r\} & \text{if }k\text{ even,}
	\end{cases}
\]
the \emph{generalized Teichm\"{u}ller codes} $\mathcal{T}_{q,k,s}$ are constructed in \cite[Sec.~3.1]{K-Diss}, \cite{K-generalized}, generalizing the \emph{Teichm\"{u}ller codes} in \cite{KZ}.
They are regular projective $R$-linear codes of length $n = 2^s\cdot\frac{q^k-1}{q-1}$ and with three non-zero homogeneous weights, scaled by the factor $1/q^{r-2}$,%
\footnote{The scaling of the homogeneous weight in \cite{K-Diss} and \cite{KZ} differs by ours by the factor $q^{r-2}$.}
\[
	w = \left(
	2^s q^k - 2^{s/2} q^{(k-1)/2},\;
	2^s q^k,\;
	2^s q^k + 2^{s/2} q^{(k-1)/2}
	\right)
\]
of frequencies
\[
	A = \left(
	\frac{1}{2}(q^k-1)(q^k + 2^{s/2}q^{(k+1)/2}),\;
	q^k-1,\;
	\frac{1}{2}(q^k-1)(q^k - 2^{s/2}q^{(k+1)/2})
	\right)\text{.}
\]
From the length we get $2^s q^k = n(q-1)$.
Thus, the sum of the weights is $S = 3\cdot 2^s q^k q^{r-2} = 3(n(q-1) + 2^s) q^{r-2} = \frac{3}{q}(b + n(q-1)q^{r-1})$ for $b = 2^s q^{r-1}$.
By Theorem~\ref{CWnewcr}  we see that the coset graph with $b$ loops on each vertices is a $3$-SWRG.

We remark that in \cite{KZ} and \cite{K-Diss}, more projective three-weight-codes over Galois rings are found, among them the dualized extended Kerdock codes $\hat{\mathcal{K}}^*_{k+1}$ and the dualized generalized Teichm\"{u}ller codes $\mathcal{T}^*_{q,k,s}$.
However, they are either not regular or the sum of the weights does not satisfy the divisibility conditions needed in the context of this paper.

\subsection{Classification of $3$-weights codes over $\F_2+u\F_2$
 with $d^\bot\ge 3$ and short length}

In a similar manner as what we did over $\Z_4,$ we classify $3$-weight codes with dual weight $d^\perp \ge 3$ over the other proper chain ring of order $4$, which is the ring $R = \F_2+u\F_2$ of dual numbers over $\F_2$.
The Gray image of such a code is always a linear binary code, and a linear binary code is the Gray image of an $R$-linear code if and only if it has a fixed-point free automorphism of order $2$ \cite{Honold-Landjev-1999-IEEETIT45[2]:700-701}.
We classify the codes in the same range of parameters as in the $\Z_4$-case, that is $n \leq 10$, $S \geq 3n$ and $S\equiv 0\mod 3$.
Since the MacWilliams identities for $R$-linear codes are the same as for $\Z_4$-linear codes (and up to doubling the length, the same as the MacWillimans identities for the linear Gray image), we can start with the same list of feasable parameters we computed for $\Z_4$-linear codes.

The feasable parameters, together with the weight distribution and the value $S - 3n \geq 0$, are listed in Table~\ref{tab_s2}.
The column $(k_1,k_2)$ lists the shapes of the realizable codes.
An entry ``$-$'' indicates that no $R$-linear code with the given parameters does exist.

\begin{table}
  \begin{center}
    $\begin{array}{lllll}
      \hline
      n & 2k_1+k_2 & \text{weights} & S-3n & (k_1,k_2)\\
      \hline
      2 & 3 & 1^{1} 2^{3} 3^{3} & 0 & - \\
      4 & 4 & 2^{1} 4^{11} 6^{3} & 0 & -\\
      4 & 5 & 2^{5} 4^{19} 6^{7} & 0 & -\\
      4 & 6 & 2^{13} 4^{35} 6^{15} & 0 & -\\
      6 & 5 & 4^{6} 6^{16} 8^{9} & 0 & (2,1) \\
      6 & 6 & 4^{18} 6^{24} 8^{21} & 0 & (3,0), (2,2) \\
      8 & 5 & 6^{6} 8^{15} 10^{10} & 0 & - \\
      8 & 6 & 6^{22} 8^{15} 10^{26} & 0 & -\\
      8 & 7 & 6^{54} 8^{15} 10^{58} & 0 & -\\
      8 & 5 & 4^{1} 8^{27} 12^{3} & 0 & (2,1) \\
      8 & 6 & 4^{5} 8^{51} 12^{7} & 0 & (3,0), (2,2) \\
      8 & 7 & 4^{13} 8^{99} 12^{15} & 0 & - \\
      10 & 5 & 8^{5} 10^{16} 12^{10} & 0 & - \\
      10 & 6 & 8^{25} 10^{8} 12^{30} & 0 & - \\
      3 & 5 & 2^{15} 4^{15} 6^{1} & 3 & (2,1) \\ % Gray image ist Parity Check code
      5 & 5 & 4^{16} 6^{12} 8^{3} & 3 & (2,1) \\
      7 & 5 & 6^{16} 8^{11} 10^{4} & 3 & - \\
      7 & 6 & 6^{42} 8^{7} 10^{14} & 3 & - \\
      7 & 7 & 4^{31} 8^{95} 12^{1} & 3 & - \\
      7 & 8 & 4^{65} 8^{187} 12^{3} & 3 & - \\
      9 & 5 & 8^{15} 10^{12} 12^{4} & 3 & (2,1) \\
      9 & 6 & 8^{43} 11^{16} 14^{4} & 6 & - \\ % linear code does not exist
      10 & 7 & 8^{62} 12^{64} 16^{1} & 6 & (3,1) \\
      10 & 8 & 8^{130} 12^{120} 16^{5} & 6 & (4,0) \\
      \hline
    \end{array}$
    \caption{Classification for $R = \F_2[X]/(X^2)$, $S \geq 3n$, $n \leq 10$}
    \label{tab_s2}
  \end{center}
\end{table}

A list of suitable generator matrices for the realizable parameters is given in Table~\ref{tab_s2g}.
\begin{table}
	\begin{center}
      $\begin{array}{llllll}
      \hline
      n & k & \text{weights} & S-3n & (k_1,k_2) & G\\
      \hline
      3 & 5 & 2^{15} 4^{15} 6^{1} & 3 & (2,1) &
      \left(\begin{smallmatrix}
	      1 & 0 & 1 \\
	      0 & 1 & 1 \\
	      0 & 0 & X
      \end{smallmatrix}\right) \\
      5 & 5 & 4^{16} 6^{12} 8^{3} & 3 & (2,1) &
      \left(\begin{smallmatrix}
	      1 & 0 & 1 & 0 & X \\
	      0 & 1 & 1 & 1 & 1 \\
	      0 & 0 & 0 & X & X
      \end{smallmatrix}\right) \\
      6 & 5 & 4^{6} 6^{16} 8^{9} & 0 & (2,1) &
      	\left(\begin{smallmatrix}
		1 & 0 & X & X+1 & 1 & 1 \\
		0 & 1 & 1 &   1 & 1 & X \\
		0 & 0 & 0 &   0 & X & X
	\end{smallmatrix}\right) \\
      6 & 6 & 4^{18} 6^{24} 8^{21} & 0 & (3,0) &
	\left(\begin{smallmatrix}
	1 & 0 & 0 & 1 & 1 & X+1 \\
	0 & 1 & 0 & 1 & X &   0 \\
	0 & 0 & 1 & 1 & 1 &   1
	\end{smallmatrix}\right) \\
      6 & 6 & 4^{18} 6^{24} 8^{21} & 0 & (2,2) &
	\left(\begin{smallmatrix}
		1 & 0 & 1 & 1 & 0 &   1 \\
		0 & 1 & 1 & X & 1 & X+1 \\
		0 & 0 & X & X & 0 &   0 \\
		0 & 0 & 0 & 0 & X &   X
	\end{smallmatrix}\right) \\
      8 & 5 & 4^{1} 8^{27} 12^{3} & 0 & (2,1) &
	\left(\begin{smallmatrix}
		1 & 0 & X & 1 & 0 & 1 & X &   1 \\
		0 & 1 & 1 & X & 1 & 1 & 1 & X+1 \\
		0 & 0 & 0 & 0 & X & X & X &   X
	\end{smallmatrix}\right) \\
      8 & 6 & 4^{5} 8^{51} 12^{7} & 0 & (3,0) &
      \left(\begin{smallmatrix}
	1 & 0 & 0 & X & X & 0 & 1 &   X \\
	0 & 1 & 0 & 0 & 1 & 1 & 0 &   1 \\
	0 & 0 & 1 & 1 & 1 & X & X & X+1
      \end{smallmatrix}\right) \\
      8 & 6 & 4^{5} 8^{51} 12^{7} & 0 & (2,2) &
      \left(\begin{smallmatrix}
	1 & 0 & X & 1 & 1 & 1 & 1 & X+1 \\
	0 & 1 & 1 & 1 & 0 & 0 & 0 &   1 \\
	0 & 0 & 0 & X & X & 0 & X &   X \\
	0 & 0 & 0 & 0 & 0 & X & X &   0
      \end{smallmatrix}\right) \\
      9 & 5 & 8^{15} 10^{12} 12^{4} & 3 & (2,1) &
      \left(\begin{smallmatrix}
	    1 & 0 & 1 & X & 0 & 1 & 1 & 1 & X+1 \\
	    0 & 1 & 1 & 1 & 1 & X & 1 & 0 &   1 \\
	    0 & 0 & 0 & 0 & X & X & X & X &   0
      \end{smallmatrix}\right) \\
      10 & 7 & 8^{62} 12^{64} 16^{1} & 6 & (3,1) &
      \left(\begin{smallmatrix}
	    1 & 0 & 0 & 1 & X & 1 & X & 0 &   1 &   0 \\
	    0 & 1 & 0 & 1 & X & 0 & 1 & 1 & X+1 & X+1 \\
	    0 & 0 & 1 & 1 & 1 & 0 & 1 & X &   1 &   1 \\
	    0 & 0 & 0 & X & X & X & X & X &   0 &   X
      \end{smallmatrix}\right) \\
      10 & 8 & 8^{130} 12^{120} 16^{5} & 6 & (4,0) &
      \left(\begin{smallmatrix}
	    1 & 0 & 0 & 0 & 1 & 0 & X & X & X+1 & 1 \\
	    0 & 1 & 0 & 0 & X & X & 0 & 1 & X+1 & 1 \\
	    0 & 0 & 1 & 0 & 0 & X & 1 & X &   1 & 1 \\
	    0 & 0 & 0 & 1 & X & 1 & X & 0 &   1 & 1
      \end{smallmatrix}\right) \\
      \hline
      \end{array}$
    \caption{Generator matrices for $R = \F_2[X]/(X^2)$, $S \geq 3n$, $n \leq 10$}
    \label{tab_s2g}
      \end{center}
\end{table}

We remark that we found further examples of $R$-linear three-weight codes with $d^\perp \geq 3$, $S-3n \geq 0$ and $3\mid S$.
Their parameters are displayed in Table~\ref{tab_s2+}
\begin{table}
  \begin{center}
    $\begin{array}{llll}
      \hline
      n & k & \text{weights} & S-3n \\
      \hline
      11 &  9 &  8^{ 162} 12^{ 312} 16^{  37} &  3 \\ %chainringdb
      11 & 10 &  8^{ 330} 12^{ 616} 16^{  77} &  3 \\ %chainringdb^{BKLC
      13 &  6 & 12^{  45} 16^{  17} 20^{   1} &  9 \\ %chainringdb
      15 &  9 & 12^{ 190} 16^{ 255} 20^{  66} &  3 \\ %chainringdb^{BKLC
      18 &  8 & 16^{ 153} 20^{  72} 24^{  30} &  6 \\ %chainringdb^{BKLC
      26 &  8 & 24^{ 172} 28^{  32} 32^{  51} &  6 \\ %chainringdb
      42 &  8 & 40^{ 189} 48^{  63} 56^{   3} & 18 \\ %BKLC
      49 &  7 & 48^{  91} 52^{  28} 56^{   8} &  9 \\ %chainringdb
      54 & 11 & 48^{ 724} 56^{1104} 64^{ 219} &  6 \\ %BKLC
      61 & 11 & 56^{ 980} 64^{ 847} 72^{ 220} &  9 \\ %BKLC
      63 & 13 & 56^{2556} 64^{4095} 72^{1540} &  3 \\ %BKLC
      \hline
    \end{array}$
    \caption{Further projective three-weight-codes over $R = \F_2[X]/(X^2)$ with $S \geq 3n$}
    \label{tab_s2+}
  \end{center}
\end{table}

%%%%%%%%%%%%%%%%%%%%%%%%%%%%%%%%%%%%%%%%%%%%%%%%%%%%%%%%%%%%%%%%%
\section{Conclusion and open problems}
In this work we have generalized from fields to rings the constructions of SWRGs from \cite{SS}. The rings considered here are $\Z_p^m,$ and more generally finite chain  rings.
The weight playing the role of the Hamming in this generalization is the homogeneous weight. This weight can defined in general over finite Frobenius rings. It would be of theoretical interest to consider three-weight codes over those rings.

We have classified short length three-weight $\Z_4$-codes leading to SWRGs. To undertake such a classification for higher lengths and other rings is a challenging open problem, both theoretically and computationally. The polynomial analogues of $\Z_p^m,$ namely the rings $\F_p[x]/(x^m),$ are worth considering.
%=============================

\end{document}